\documentclass[11pt]{amsart}

\usepackage[top=1in,bottom=1in,left=1in,right=1in,heightrounded]{geometry}
\usepackage{setspace}
\usepackage[parfill]{parskip}
\allowdisplaybreaks

\usepackage{amsmath,amsthm,amssymb} 
\usepackage{mathtools}
\usepackage{mathrsfs}
\usepackage{stmaryrd}

\usepackage{thmtools}
\usepackage{thm-restate}
\usepackage{comment}

\usepackage{tikz}
\usepackage{tikz-cd}
\usetikzlibrary{decorations.pathmorphing,decorations.markings,patterns,hobby,arrows.meta}
\tikzset{>=stealth}
\usepackage{pgfplots}
    \pgfplotsset{compat=1.18}
\usepackage{graphicx}

\usepackage{array}
\usepackage{makecell}

\usepackage{enumitem}

\usepackage{caption} 
\usepackage{float}

\usepackage{fancyhdr}

\usepackage{xcolor,colortbl}
\usepackage{soul}

\usepackage{hyperref}
\usepackage[capitalise]{cleveref}
\usepackage{url}
\hypersetup{
   colorlinks,
   citecolor=blue,
   filecolor=blue,
   linkcolor=blue,
   urlcolor=blue
}

\usepackage[numbers]{natbib}
\usepackage{lmodern}
\usepackage{mathpazo}
\usepackage[cal=boondox,scr=boondox]{mathalfa}

\AtBeginDocument{
  \DeclareSymbolFont{AMSb}{U}{msb}{m}{n}
  \DeclareSymbolFontAlphabet{\mathbb}{AMSb}}

\DeclareMathAlphabet{\mathbx}{U}{BOONDOX-ds}{m}{n}
\SetMathAlphabet{\mathbx}{bold}{U}{BOONDOX-ds}{b}{n}
\DeclareMathAlphabet{\mathbbx}{U}{BOONDOX-ds}{b}{n}

\SetMathAlphabet{\mathcal}{bold}{U}{dutchcal}{b}{n}
\DeclareMathAlphabet{\mathbcal}{U}{dutchcal}{b}{n}

\newtheorem{theorem}{Theorem}[section]
\newtheorem{lemma}[theorem]{Lemma}
\newtheorem{corollary}[theorem]{Corollary}
\newtheorem{proposition}[theorem]{Proposition}
\newtheorem*{theproblem*}{Problem Statement}

\theoremstyle{definition}

\newtheorem{remark}[theorem]{Remark}
\newtheorem{definition}[theorem]{Definition}
\newtheorem{question}[theorem]{Question}

\newcommand{\zz}{\mathbx Z}   
\newcommand{\ff}{\mathbx F}   
\newcommand{\kk}{\mathbx k}   

\newcommand{\oo}{\mathscr O}     

\newcommand{\ord}{\mu\text{-}\mathrm{ord}} 

\renewcommand{\epsilon}{\varepsilon}
\renewcommand{\phi}{\varphi}

\renewcommand{\geq}{\geqslant}
\renewcommand{\leq}{\leqslant}

\renewcommand{\ss}{\textnormal{ss}}  

\numberwithin{equation}{section}

\colorlet{deewangcolor}{cyan!50}

\colorlet{emeraldcolor}{blue!30}

\colorlet{heidicolor}{magenta!50}

\colorlet{miacolor}{yellow!90}

\colorlet{sandracolor}{green!40!yellow}

\colorlet{stevencolor}{orange!60}

\title{The Ekedahl-Oort and Newton stratification of\\ the $\boldsymbol{\mathsf{GU}(3,2)}$ Shimura variety}

\author{Emerald Andrews} 
\address{Department of Mathematics and Computer Science, Washington College, Chestertown, MD 21620, USA}
\email{estacy2@washcoll.edu}

\author{Deewang Bhamidipati} 
\address{Department of Mathematics and Statistics, Carleton College, Northfield, MN 55057, USA} 
\email{bdeewang@carleton.edu}

\author{Maria Fox} 
\address{Department of Mathematics, Oklahoma State University, Stillwater, OK 74078, USA}
\email{maria.fox@okstate.edu}

\author{Heidi Goodson} 
\address{Department of Mathematics, Brooklyn College and The Graduate Center, City University of New York, Brooklyn, NY 11210, USA}
\email{heidi.goodson@brooklyn.cuny.edu}

\author{\\ Steven R. Groen} 
\address{Korteweg-de Vries Institute for Mathematics, University of Amsterdam, Amsterdam, Netherlands}
\email{s.r.groen@uva.nl}

\author{Sandra Nair} 
\address{Department of Mathematics, Colorado State University, Fort Collins, CO 80523, USA}
\email{sandra.nair@colostate.edu}

\begin{document}

\begin{abstract}
    This paper concerns the characteristic-$p$ fibers of $\mathsf{GU}(3,2)$ Shimura varieties. Such Shimura varieties parametrize abelian varieties in characteristic $p$ of dimension $5$ with an action of signature $(3,2)$ by an order in an imaginary quadratic field in which $p$ is inert. We completely describe the interaction of two stratifications of these Shimura varieties: the Ekedahl-Oort stratification, based on the isomorphism class of the $p$-torsion subgroup scheme, and the Newton stratification, based on the isogeny class of the $p$-divisible group. We identify which Ekedahl-Oort and Newton strata intersect.
\end{abstract}
\maketitle
\section{Introduction}\label{sec:intro}

Shimura varieties of PEL type are moduli spaces of abelian varieties with certain additional structures. A classical example is the modular curve, a dimension 1 PEL type Shimura variety that parameterizes elliptic curves. There are two stratifications on the characteristic-$p$ fiber of PEL type Shimura varieties arising from their moduli interpretation:

\begin{enumerate}
    \item The \textbf{Ekedahl-Oort stratification}, based on the isomorphism classes of the $p$-torsion group schemes of the parametrized abelian varieties;
    \item The \textbf{Newton stratification}, based on the isogeny classes of the $p$-divisible groups of the parametrized abelian varieties.
\end{enumerate}

These stratifications coincide on the modular curve, differentiating ordinary and supersingular elliptic curves into separate strata. However, in a more general setting where the PEL type Shimura variety parametrizes higher-dimensional abelian varieties, the interactions between the Ekedahl-Oort and Newton stratifications are more complicated; c.f. the results in \cite{ViehmannWedhorn2013}. A particularly concrete setting to study interactions between these two stratifications is that of \emph{unitary Shimura varieties} $\mathcal{M}(q-b, b)$. These are moduli spaces of abelian varieties of dimension $q$ with an action of an order in an imaginary quadratic field $K$ that meets the ``signature $(q-b, b)$" condition. In this paper, we assume $p$ is inert in $K$ and describe completely how the Ekedahl-Oort stratification and the Newton stratification interact on $\mathcal{M}(3,2)$.

The interactions between Ekedahl-Oort and Newton stratifications of unitary Shimura varieties are well-established in some signatures. An easy case is signature $(q,0)$, since the unitary Shimura varieties with this signature are zero-dimensional, with a single Ekedahl-Oort stratum and a single Newton stratum. Those of signature $(q-1,1)$ have been extensively studied, for example, in \cite{bultel2006congruence,VollaardWedhorn}. Signature $(2,2)$ unitary Shimura varieties are studied in  \cite{howard2014supersingular}. Interestingly, \cite[Theorem A]{goertzhe} shows that these are all the signatures for which the Shimura variety is of Coxeter type, meaning every Newton stratum is a union of Ekedahl-Oort strata. 

Recently, the Ekedahl-Oort stratification of signature $(q-2,2)$ unitary Shimura varieties was studied in \cite{shimada2024supersingular} with a primary focus on the supersingular locus, the unique closed Newton stratum. In our paper, we focus on the signature $(3,2)$ and give a complete description of the interaction between the Ekedahl-Oort strata and all of the Newton strata. We do this using a variety of techniques, some of which were developed by the authors in \cite{RNT}, and others developed here. Many of our techniques can be applied to more general signatures, and we hope that our work will inspire further developments in this area of research.  

In Section~\ref{s:p-rank}, we use the fact that the Ekedahl-Oort and the Newton stratifications are both refinements of the $p$-rank stratification, restricting how the strata interact. In Section~\ref{s:SMV}, we apply the map from $\mathcal{M}(3,2)$ to the Siegel modular variety $\mathcal{A}_5$, forgetting the unitary structure. We employ various tools available for $\mathcal{A}_5$, such as generic slopes and minimal Ekedahl-Oort strata, to gain information about $\mathcal{M}(3,2)$. In Section~\ref{s:construction} we explicitly construct a point in the intersection of an Ekedahl-Oort stratum and a Newton stratum to settle the last remaining question. We completely describe the interaction between the two stratifications of $\mathcal{M}(3,2)$ in Theorem~\ref{thm:main}, 

\subsection*{Acknowledgments}

This project started at the Rethinking Number Theory workshop in June 2022. We thank the organizers for the workshop, and are grateful for the supportive, collaborative environment this workshop provided. The workshop was supported by the Number Theory Foundation, the American Institute of Mathematics, and the University of Wisconsin-Eau Claire.  Additionally, we thank Rachel Pries for helpful feedback on this project's initial results.

M.F. was supported in part by NSF MSPRF Grant 2103150. H.G. was supported by NSF grant DMS-2201085. The authors thank AIM and NSF for their support through the valuable American Institute of Mathematics (AIM) SQuaRE program, which was crucial in this collaboration. 


\section{Background}\label{sec:background}

Unitary Shimura varieties are moduli spaces of abelian varieties equipped with extra structure, including an action of an order in an imaginary quadratic field $K$. We fix a prime $p > 2$ which is inert in $K$. The main object of study for this paper is the char $p$ unitary Shimura variety $\mathcal{M}(3,2)$. This is a quasi-projective variety of dimension $6$ defined over $\ff_{p^2}$, the residue field of $K$ at $p$. It arises as the char $p$ fiber of an integral unitary Shimura variety associated to a PEL datum, constructed by Kottwitz \cite{Ko}; see \cite{vollaard} or \cite[Section 2.1]{RNT} for a detailed definition.

Let $\kk$ be an algebraic closure of $\ff_{p^2}$. The $\kk$-points of $\mathcal{M}(3,2)$ correspond to isomorphism classes of tuples $(A, \iota, \lambda, \xi)$, where $A$ is an abelian variety over $\kk$ of dimension $5$, with a prime-to-p quasi-polarization $\lambda$, an action $\iota$ of signature $(3,2)$ by an order in an imaginary quadratic field $K$ in which $p$ is inert, and level structure $\xi$.  Two tuples $(A, \iota, \lambda, \xi)$ and $(A', \iota', \lambda', \xi')$ are isomorphic if there exists a prime-to-$p$ isogeny from $A$ to $A'$, commuting with the action of $\oo_K \otimes_{\zz} \zz_{(p)}$, mapping $\xi$ to $\xi'$ and $\lambda$ to a $\mathbb{Z}_{(p)}^\times$-multiple of $\lambda'$. 

In our work, we are interested in characterizing the interaction between two stratifications of the unitary Shimura variety $\mathcal{M}(3,2)$ called the Ekedahl-Oort stratification and the Newton stratification. We define these and set up the problem of interest in the remainder of this section.

\subsection{Ekedahl-Oort strata} \label{ss:EO}

The Ekedahl-Oort stratification partitions the special fibers of PEL type Shimura varieties on the basis of isomorphism classes of the $p$-torsion group schemes of the underlying abelian varieties with extra structures. See \cite{ViehmannWedhorn2013} for an excellent exposition of Ekedahl-Oort strata of PEL type Shimura varieties. For unitary Shimura varieties $\mathcal{M}(a,b)$, two $\kk$-points $(A,\iota, \lambda, \xi)$ and $(A', \iota', \lambda', \xi')$ of $\mathcal{M}(a,b)$ are in the same \textbf{Ekedahl-Oort stratum} if there exists an isomorphism between $A[p]$ and $A'[p]$ that respects the induced action and polarization.

It is outlined in \cite{Moonen2001} that Ekedahl-Oort strata of PEL type Shimura varieties correspond bijectively to cosets of a relevant Weyl group (which depends on the PEL datum). In the case of $\mathcal{M}(a,b)$ with $a+b=q$, let $\mathfrak{S}_q$ be the symmetric group on $q$ elements. Given a permutation $\omega \in \mathfrak{S}_q$, its \textbf{length} $\ell(\omega)$ is defined to be the length of the shortest expression that writes $\omega$ as a product of simple reflections. Equivalently, $\ell(\omega)$ equals the cardinality of the set 
\begin{equation} \label{eq:length}
    \mathrm{Inv}(\omega) := \{ (i,j) \; | \; i<j \text{ and } \omega(i) > \omega(j) \}.
\end{equation}
In \cite[Theorem 6.7]{Moonen2001}, Moonen proves a bijection between the set of Ekedahl-Oort strata of $\mathcal{M}(a,b)$ and the cosets in $(\mathfrak{S}_a \times \mathfrak{S}_b) \setminus \mathfrak{S}_q$. Each coset has a unique representative of minimal length and these minimal length representatives are given explicitly by
\begin{equation*} \label{eq:gammauv}
    \gamma_{u_1,\ldots ,u_b} := (b, b+1, \ldots ,u_b) \cdots (2, 3, \ldots ,u_2)(1,2,\ldots,u_1)
\end{equation*}
for $1\leq u_1 < u_2 < \ldots < u_b\leq q$, by \cite[Proposition~3.2]{RNT}. Finally, define
\begin{equation*} \label{eq:W(a,b)}
    \mathbf{W}(a,b) := \{ \gamma_{u_1,\ldots , u_b} \; | \; 1 \leq u_1 < \ldots < u_b\leq q\}
\end{equation*}
so that \cite[Theorem 6.7]{Moonen2001} gives a bijection between $\mathbf{W}(a,b)$ and the set of Ekedahl-Oort strata of $\mathcal{M}(a,b)$. Under this correspondence, the length of the permutation equals the dimension of the Ekedahl-Oort stratum, where $\ell(\gamma_{u_1,\ldots ,u_b}) = \sum_{i=1}^b(u_i - i)$ \cite[Lemma 3.1]{RNT}.

In particular, \cite[Corollary 3.6]{RNT} provides that the Ekedahl-Oort strata of $\mathcal{M}(3,2)$ are represented by the permutations 
$$\gamma_{1,2},\gamma_{1,3},\gamma_{1,4},\gamma_{1,5},\gamma_{2,3},\gamma_{2,4},\gamma_{2,5},\gamma_{3,4}, \gamma_{3,5}, \gamma_{4,5} \in \mathfrak{S}_5.$$
Given a permutation $\gamma_{u,v}$, let $\mathcal{M}(3,2)_{\gamma_{u,v}}$ be the Ekedahl-Oort stratum of $\mathcal{M}(3,2)$ that corresponds to $\gamma_{u,v}$ via \cite[Theorem 6.7]{Moonen2001}; observe that $\dim \mathcal{M}(3,2)_{\gamma_{u,v}} = \ell(\gamma_{u,v}) = u + v - 3$.

\subsection{Newton strata} \label{ss:Newton}

The Newton stratification partitions the special fiber of PEL type Shimura varieties on the basis of the isogeny classes of the $p$-divisible groups of the underlying abelian varieties with extra structures. More precisely, two points $(A,\iota, \lambda, \xi)$ and $(A',\iota', \lambda' ,\xi')$ are in the same Newton stratum if there exists an isogeny between the $p$-divisible groups $A[p^\infty]$ and $A'[p^\infty]$ that respects the induced action and polarization. 

By the Dieudonn\'{e}-Manin theorem, any $p$-divisible group decomposes into simple isoclinic factors up to isogeny. For coprime non-negative integers $m$ and $n$, let $G_{m,n}$ be an isoclinic $p$-divisible group of dimension $m$, codimension $n$ and height $m+n$. Up to isogeny, it decomposes as:
\begin{equation} \label{eq:DDManin}
A[p^\infty] \sim \prod_{i=1}^r G_{m_i,n_i}^{\ell_i}.
\end{equation}
Here $\ell_i$ is called the multiplicity of $G_{m_i,n_i}$. For the $p$-divisible groups on both sides of Equation~\eqref{eq:DDManin} to have the same height and dimension, we must have
\[ \sum_{i=1}^r \ell_im_i = \sum_{i=1}^r \ell_in_i=q. \]
One then defines the \textbf{slopes} $\alpha_i=\frac{m_i}{m_i+n_i}$. The multiset of slopes of $A$ is written as $\alpha_A = \left[\alpha_1^{\ell_1} , \ldots ,\alpha_r^{\ell_r} \right]$.
After relabeling factors in Equation~\eqref{eq:DDManin} if necessary, we may assume $0 \leq \alpha_1 < \cdots < \alpha_r \leq 1$. Since $A$ is principally polarized, $G_{m_i,n_i}$ and $G_{n_i,m_i}$ occur with the same multiplicity. Therefore we have $\alpha_i+\alpha_{r+1-i} = 1$ for all $1 \leq i \leq r$. Connecting line segments with the slopes in $\alpha_A$ gives a polygon from $(0,0)$ to $(2q,q)$. We shall refer to this polygon as the \textbf{Newton polygon} of $A$.

Given a $\kk$-point $(A,\iota,\lambda,\xi)$ on $\mathcal{M}(a,b)$, there are additional restrictions on the Newton polygon of $A$ (for details, see \cite{rapoport1996classification}).

For $\mathcal{M}(3,2)$, there are four possible Newton polygons:
\begin{enumerate}
    \item The Newton polygon with slopes $\beta_{\ss} = \left [\frac{1}{2}^5 \right ]$;
    \item The Newton polygon with slopes $\beta_1 = \left [\frac{1}{4}, \frac{1}{2}, \frac{3}{4}   \right]$;
    \item The Newton polygon with slopes $\beta_2 = \left[ \frac{0}{1}^2, \frac{1}{2}^3, \frac{1}{1}^2  \right]$;
    \item The Newton polygon with slopes $\beta_{\text{max}} = \left[ \frac{0}{1}^4, \frac{1}{2}, \frac{1}{1}^4 \right].$ 
\end{enumerate}

Given a multiset $\alpha$ of slopes, we denote by $\mathcal{M}(3,2)^\alpha$ the Newton stratum corresponding to $\alpha$ via the Dieudonn\'{e}-Manin theorem. The unique closed Newton stratum $\mathcal{M}(3,2)^{\ss} = \mathcal{M}(3,2)^{\beta_{\ss}}$ is known as the \textbf{supersingular locus} or \textbf{basic locus}. On the other extreme, the open and dense Newton stratum $\mathcal{M}(3,2)^{\ord} = \mathcal{M}(3,2)^{\beta_{\mathrm{max}}}$ is known as the \textbf{$\mathbold{\mu}$-ordinary locus}.

Recall that our goal is to describe the interaction between the Ekedahl-Oort stratification and the Newton stratification of $\mathcal{M}(3,2)$. The following question captures this more precisely.

\begin{question} \label{q}
Given a permutation $\gamma_{u,v} \in \mathbf{W}(3,2)$ and a multiset of slopes $\alpha$, does the Ekedahl-Oort stratum $\mathcal{M}(3,2)_{\gamma_{u,v}}$ intersect the Newton stratum $\mathcal{M}(3,2)^\alpha$?
\end{question}

A complete answer to Question~\ref{q} is given in Theorem~\ref{thm:main}. In Sections \ref{s:p-rank} -- \ref{s:construction} we describe a variety of techniques used to describe the interaction between the stratifications.

\section{Arguments based on \texorpdfstring{$p$}{}-ranks} \label{s:p-rank}

Our goal is to understand the interaction between the Ekedahl-Oort stratification and the Newton stratification of $\mathcal{M}(3,2)$. The first tool that we use to achieve this goal is the $p$-rank. Recall that $\kk$ is an algebraically closed field of characteristic $p$. The \textbf{$\boldsymbol{p}$-rank} of an abelian variety $A$ over $\kk$ is the integer $0 \leq f_A \leq \dim A$ such that
\begin{equation} \label{eq:p-rank}
A[p](\kk) \cong (\zz/p \zz)^{f_A}.
\end{equation}
Two $\kk$-points $(A,\iota,\lambda, \xi)$ and $(A',\iota', \lambda',\xi')$ of $\mathcal{M}(3,2)$ are in the same \textbf{$\boldsymbol{p}$-rank stratum} if $A$ and $A'$ have the same $p$-rank. For $0 \leq f \leq 5$, we denote by $\mathcal{M}(3,2)^{(f_A=f)}$ the $p$-rank $f$ stratum of $\mathcal{M}(3,2)$. The closure of a $p$-rank stratum is the union of all smaller $p$-rank strata.

We note that the Ekedahl-Oort stratification is a refinement of the $p$-rank stratification. In particular, if $M$ is the mod-$p$ Dieudonn\'{e} module corresponding to an Ekedahl-Oort stratum, then the $p$-rank of $M$ equals the $\kk$-dimension of the largest subspace of $M$ on which $F$ acts bijectively. 

Since the $p$-rank is an invariant under isogenies, it follows that the Newton stratification is also a refinement of the $p$-rank stratification. In particular, $f_A$ equals the multiplicity of the slope $0$ in the Newton polygon of $A$. It follows from the list of Newton polygons that the non-empty $p$-rank strata of $\mathcal{M}(3,2)$ are $\mathcal{M}(3,2)^{(f_A=0)}$, $\mathcal{M}(3,2)^{(f_A=2)}$ and $\mathcal{M}(3,2)^{(f_A=4)}$.

Since the Ekedahl-Oort stratification and the Newton stratification both refine the $p$-rank stratification, an Ekedahl-Oort stratum and a Newton stratum can only intersect if they are contained in the same $p$-rank stratum. This observation has significant consequences on Question~\ref{q}.

\begin{lemma} \label{l:p-rank}
We have  
\begin{align}
    \mathcal{M}(3,2)^{(f_A=4)} & = \mathcal{M}(3,2)^{\ord} = \mathcal{M}(3,2)_{\gamma_{4,5}}  ,\label{eq:p-rank 4} \\ 
    \mathcal{M}(3,2)^{(f_A=2)} &= \mathcal{M}(3,2)^{\beta_2} = \mathcal{M}(3,2)_{\gamma_{3,5}} \cup \mathcal{M}(3,2)_{\gamma_{2,5},} \label{eq:p-rank 2} \\
    \mathcal{M}(3,2)^{(f_A=0)} &= \mathcal{M}(3,2)^{\beta_1} \cup \mathcal{M}(3,2)^{\ss}= \bigcup_{v<5 \text{ or } u=1} \mathcal{M}(3,2)_{\gamma_{u,v}}. \label{eq:p-rank 0}
\end{align}
Furthermore, the following closure relations hold:
\begin{align}
    \overline{\mathcal{M}(3,2)^{\ord}} &= \overline{\mathcal{M}(3,2)_{\gamma_{4,5}}}=\mathcal{M}(3,2), \label{eq:p-rank 4 closure} \\
    \overline{\mathcal{M}(3,2)^{\beta_2}} &= \overline{\mathcal{M}(3,2)_{\gamma_{3,5}}} = \mathcal{M}(3,2)^{(f_A=2)} \cup \mathcal{M}(3,2)^{(f_A=0)}, \label{eq:p-rank 2 closure} \\
    \overline{\mathcal{M}(3,2)^{\beta_1}} &= \overline{\mathcal{M}(3,2)_{\gamma_{3,4}}} = \mathcal{M}(3,2)^{(f_A=0)}. \label{eq:p-rank 0 closure}
\end{align}
\end{lemma}
\begin{proof}
Equations \eqref{eq:p-rank 4}, \eqref{eq:p-rank 2} and \eqref{eq:p-rank 0} follow immediately from computing the $p$-ranks of the Ekedahl-Oort strata, for instance by constructing the standard objects introduced in \cite[4.9]{Moonen2001}. Alternatively, Equation~\eqref{eq:p-rank 4} follows from \cite{MoonenST} and it follows from \cite[Proposition 5.6]{RNT} that the Ekedahl-Oort strata with positive $p$-rank are precisely the strata $\mathcal{M}(3,2)_{\gamma_{u,5}}$ with $u>1$. 

As for the closure relations, observe that the density of the $\mu$-ordinary locus directly implies Equation~\eqref{eq:p-rank 4 closure}. Furthermore, Equation~\eqref{eq:p-rank 2 closure} follows from the purity of $p$-rank strata and the fact that $\mathcal{M}(3,2)_{\gamma_{3,5}}$ is the only Ekedahl-Oort stratum in $\mathcal{M}(3,2)^{(f_A=2)}$ that has the same dimension as $\mathcal{M}(3,2)^{(f_A=2)}$. Finally, Equation~\eqref{eq:p-rank 0 closure} follows using the same argument.
\end{proof}

Lemma~\ref{l:p-rank} yields significant progress towards answering Question~\ref{q}. It gives a complete answer for all Ekedahl-Oort strata and Newton strata of positive $p$-rank, since $\mathcal{M}(3,2)^{(f_A=4)}$ and $\mathcal{M}(3,2)^{(f_A=2)}$ contain only one Newton stratum. We henceforth restrict our attention to Ekedahl-Oort strata and Newton strata of $p$-rank $0$.

\section{Arguments based on the Siegel modular variety} \label{s:SMV}

\subsection{The forgetful map} \label{ss:forget}

In this section, we make further progress towards answering Question~\ref{q} by exploiting the \textbf{forgetful map} from the unitary Shimura variety $\mathcal{M}(a,b)$ to the Siegel modular variety $\mathcal{A}_q$. The forgetful map is defined as
\begin{align*}
    \Psi_{a,b}: \mathcal{M}(a,b) &\to \mathcal{A}_q \\
    (A,\iota,\lambda,\xi) &\mapsto (A,\lambda,\xi),
\end{align*}

where $a+b=q$. The forgetful map $\Psi_{a,b}$ induces a map on Ekedahl-Oort strata. More precisely, let $\mathbf{W}_q$ be the set of permutations $\omega \in \mathfrak{S}_{2q}$ satisfying
\[\omega^{-1}(1) < \omega ^{-1}(2) < \cdots < \omega^{-1}(q) \quad \text{and} \quad \omega(i) + \omega(2q+1-i) = 2q+1.\]

By \cite[3.6]{Moonen2001},  the Ekedahl-Oort strata of $\mathcal{A}_q$ correspond bijectively to elements of $\mathbf{W}_q$. Thus $\Psi_{a,b}$ induces a map
$$\psi_{a,b}: \mathbf{W}(a,b) \to \mathbf{W}_q.$$ Given $\omega \in \mathbf{W}_q$, let $\mathcal{A}_{q,\omega}$ denote the corresponding Ekedahl-Oort stratum of $\mathcal{A}_q$.

In \cite[Theorem 5.2]{RNT}, the map $\psi_{q-2,2}$ is completely described for general $q\geq 2$. We can leverage these results by setting $q=5$. For $\gamma_{u,v} \in \mathbf{W}(a,b)$, let us denote $\omega_{u,v}:=\psi_{3,2}(\gamma_{u,v})\in \mathbf{W}_q$. We first consider the cases when $\mathcal{A}_{5,\omega_{u,v}}$ is contained in the supersingular locus of $\mathcal{A}_5$.

\begin{lemma} \label{l:Hoeve}
The Ekedahl-Oort strata $\mathcal{M}(3,2)_{\gamma_{1,2}}$ and $\mathcal{M}(3,2)_{\gamma_{1,3}}$ are contained in $\mathcal{M}(3,2)^{\ss}$.
\end{lemma}
\begin{proof}
This is the case $q=5$ of \cite[Corollary 5.4]{RNT}. The explicit description of $\psi_{3,2}$ is used to show that $\omega_{1,2}(3) = \omega_{1,3}(3)=3$, so $\mathcal{A}_{5,\omega_{1,2}}$ and $\mathcal{A}_{5,\omega_{1,3}}$ are contained in the supersingular locus of $\mathcal{A}_5$ by \cite{Hoeve2009}. This implies that $\mathcal{M}(3,2)_{\gamma_{1,2}}$ and $\mathcal{M}(3,2)_{\gamma_{1,3}}$ are contained in $\mathcal{M}(3,2)^{ss}$. 
\end{proof}

\subsection{Minimal Ekedahl-Oort strata} \label{ss:minimal}

An Ekedahl-Oort stratum of a PEL type Shimura variety is called \textbf{minimal} if all the parameterized abelian varieties in the Ekedahl-Oort stratum have isomorphic $p$-divisible groups. A minimal Ekedahl-Oort stratum is completely contained in one Newton stratum. While the existence and uniqueness of minimal Ekedahl-Oort strata in Newton strata of unitary Shimura varieties is in general unknown, every Newton stratum of $\mathcal{A}_q$ contains a unique minimal Ekedahl-Oort stratum \cite{Oort05minimal,Oort05simple}. The Dieudonn\'{e} modules of these minimal Ekedahl-Oort strata are constructed explicitly in \cite[5.3]{DeJongOort}. 

\begin{lemma} \label{l:gamma24}
The Ekedahl-Oort stratum $\mathcal{A}_{5,\omega_{2,4}}$ is the minimal Ekedahl-Oort stratum of $\mathcal{A}_5$ with slope sequence $\beta_1 = \left[ \frac{1}{4}, \frac{1}{2} , \frac{3}{4} \right]$. Thus, the Ekedahl-Oort stratum $\mathcal{M}(3,2)_{\gamma_{2,4}}$ is completely contained in $\mathcal{M}(3,2)^{\beta_1}$.
\end{lemma}
\begin{proof}
This is an application of \cite[Proposition 5.11]{RNT} with $(n_1,n_2,n_3)=(1,1,3)$. First, \cite[Theorem 5.2]{RNT} gives
$$ \omega_{2,4} =\psi_{3,2}(\gamma_{2,4}) = (2,6,8,4) (3,7,9,5).$$
Additionally, we construct the minimal Ekedahl-Oort stratum corresponding to the slope sequence $\alpha= \left[\frac{1}{4}, \frac{1}{2}, \frac{3}{4}  \right]$. In the notation of \cite[5.3]{DeJongOort}, its mod-$p$ Dieudonn\'{e} module is given by 
$$M_\alpha=M_{1,3} \oplus M_{1,1} \oplus M_{3,1}.$$
Finally, one verifies via \cite[3.6]{Moonen2001} that $M_\alpha$ corresponds to the permutation $\omega_{2,4}$. 

Thus we have proved that $\mathcal{A}_{5,\omega_{2,4}}$ is the minimal Ekedahl-Oort stratum of the Newton stratum $\mathcal{A}_5^{\beta_1}$. This implies that $\mathcal{M}(3,2)_{\gamma_{2,4}}$ must be contained in $\mathcal{M}(3,2)^{\beta_1}$.
\end{proof}

\subsection{Generic slopes} \label{ss:generic}

In this section, we use results of \cite{Harashita_slope, Harashita_polygon} to compute slopes of generic Newton polygons of strata in the Siegel modular variety $\mathcal A_q$ to obtain results for Ekedahl-Oort strata for the unitary Shimura variety $\mathcal M(q-b,b)$. For our purposes, it suffices to focus on the techniques in \cite{Harashita_slope} for computing the \emph{first slope} of the generic Newton polygon for a given stratum, since it follows from \cite[Theorem 4.1]{Harashita_slope} that all other polygons for the Ekedahl-Oort stratum must have first slope greater than or equal to this value. We will use this fact to shed light on some interactions between the Ekedahl-Oort and Newton stratifications in signature $(3,2)$.

\begin{definition}
A \textbf{final sequence} is a function $\phi: \{0, 1, \ldots ,2q \} \to \{0, 1, \ldots , q\}$, satisfying the following conditions:
\begin{itemize}
    \item $\phi(0)=0$;
    \item $\phi(i-1) \leq \phi(i) \leq \phi(i-1)+1$ for $1\leq i \leq 2q$;
    \item $\phi(2q-i) = q-i+\phi(i)$ for $0 \leq i \leq 2q$.
\end{itemize}
In the language of \cite[Definition 2.2]{Harashita_slope}, this is a final sequence \emph{stretched from an elementary sequence}. We will write a final sequence down as an ordered tuple $\phi=[\phi(1), \ldots ,\phi(2q)]$. 
\end{definition}
Given an element $\omega \in \mathfrak{S}_{2q}$, we can construct its corresponding final sequence $\phi$ as follows:
\begin{itemize}
    \item Set $\phi(0)=0$;
    \item For $1\leq i \leq 2q$, if $\omega(i)>q$, set $\phi(i)=\phi(i-1)+1$;
    \item For $1\leq i \leq 2q$, if $\omega(i)\leq q$, set $\phi(i)=\phi(i-1)$.
\end{itemize} 
Given $\gamma_{u,v}\in \mathbf{W}(3,2)$, let $\phi_{u,v}$ denote the final sequence corresponding to $\omega_{u,v}=\psi_{3,2}(\gamma_{u,v})$.     

We now demonstrate how to compute the first slope of the Newton polygon of a generic point in a given Ekedahl-Oort stratum from the above construction. Note that, for ease of notation, our exposition differs slightly from that of \cite{Harashita_slope}. We first define a map on the set $S=\{1,\ldots,2q\}$ based on the image of the final sequence $\phi$. Let $\tilde\phi$ be the map defined by
\begin{equation}\label{eqn:Harashita_tildephi}
    \tilde\phi(i)=\begin{cases}
    \phi(i) & \text{ if } \phi(i)\not=0,\\
    q+i & \text{ otherwise.}
\end{cases}
\end{equation}

\begin{definition}[{\citenum{Harashita_slope}, Definition~3.1}]\label{def:Harashita_slope}
    Let $\mathcal D =\cap_{j=1}^\infty \tilde\phi^j(S)$ and let $\mathcal C=\mathcal D \cap \{q+1,\ldots, 2q\}$. The \emph{generic first slope} associated with $\phi$ is 
    $$\lambda_\phi = \frac{\#\mathcal C}{\#\mathcal D}.$$
\end{definition}

With this setup, we are able to prove the following results.

\begin{lemma} \label{l:gamma14}
The generic first slope of the Ekedahl-Oort stratum $\mathcal{A}_{q,\omega_{1,4}}$ is $\frac{2}{5}$.
\end{lemma}
\begin{proof}
First, \cite[Theorem 5.2]{RNT} gives $\omega_{1,4} = \psi_{3,2}(\gamma_{1,4})= (3,6,4) (5,7,8)$ and hence $$\phi_{1,4}= [0,0,1,1,2,2,3,3,4,5].$$ It is computed in \cite[Example 3.19.(3)]{Harashita_slope} that the generic first slope is $\frac{2}{5}$.
\end{proof}

\begin{lemma} \label{l:gamma15}
The generic first slope of the Ekedahl-Oort strata $\mathcal{A}_{q,\omega_{1,5}}$ and  $\mathcal{A}_{q,\omega_{2,3}}$ is $\frac{1}{3}$.

\end{lemma}
\begin{proof}
From \cite[Theorem 5.2]{RNT} we have that $\omega_{1,5} = \omega_{2,3} = (2,6,4,3) (5,7,8,9)$ and hence 
$$\phi_{1,5}=\phi_{2,3} = [0,1,1,1,2,2,3,4,4,5].$$
Following Equation \eqref{eqn:Harashita_tildephi} and Definition \ref{def:Harashita_slope}, we compute $\mathcal{D}=\{1, 2, 6\}$ and $\mathcal{C}=\{6\}$. Therefore the generic first slope is $\frac{\# \mathcal{C}} {\# \mathcal{D}} = \frac{1}{3}$.
\end{proof}

\begin{remark}
    An abstract way to view the phenomenon $\omega_{1,5}=\omega_{2,3}$ is that the Ekedahl-Oort strata $\mathcal{M}(3,2)_{\gamma_{1,5}}$ and $\mathcal{M}(3,2)_{\gamma_{2,3}}$ have the same underlying $p$-torsion group scheme with polarization, but a different action of $\ff_{p^2}$.
\end{remark}

\begin{corollary} \label{c:generic}
The Ekedahl-Oort strata $\mathcal{M}(3,2)_{\gamma_{1,4}}$, $\mathcal{M}(3,2)_{\gamma_{1,5}}$ and $\mathcal{M}(3,2)_{\gamma_{2,3}}$ are contained in $\mathcal{M}(3,2)^{\ss}$.
\end{corollary}
\begin{proof}
It follows from Lemma~\ref{l:gamma14} that the first slope of any Newton polygon occurring on $\mathcal{M}(3,2)_{\gamma_{1,4}}$ is at least $\frac{2}{5}$. Similarly, Lemma~\ref{l:gamma15} posits that the first slope of any Newton polygon occurring on $\mathcal{M}(3,2)_{\gamma_{1,5}}$ or $\mathcal{M}(3,2)_{\gamma_{2,3}}$ is at least $\frac{1}{3}$. On the other hand, the only Newton polygon on $\mathcal{M}(3,2)$ whose first slope is at least $\tfrac{1}{3}$ is the supersingular Newton polygon $\beta_{\ss} = \left[ \tfrac{1}{2}^5  \right]$. Thus $\beta_{\ss}$ is the only Newton polygon that can occur on the Ekedahl-Oort strata in question. In other words, these Ekedahl-Oort strata are contained in the supersingular locus $\mathcal{M}(3,2)^{\ss}$.
\end{proof}

\begin{remark} \label{r: Hoeve vs Harashita}
It is insightful to compare Corollary~\ref{c:generic} to Lemma~\ref{l:Hoeve}. As opposed to $\mathcal{A}_{5,\omega_{1,2}}$ and $\mathcal{A}_{5,\omega_{1,3}}$, the Ekedahl-Oort strata $\mathcal{A}_{5,\omega_{1,4}}$, $\mathcal{A}_{5,\omega_{1,5}}$ and $\mathcal{A}_{5,\omega_{2,3}}$ are not contained in the supersingular locus of $\mathcal{A}_5$. However, the non-supersingular Newton polygons that occur on these Ekedahl-Oort strata of $\mathcal{A}_5$ are not Newton polygons that occur on $\mathcal{M}(3,2)$. In terms of our moduli interpretation, there are non-supersingular abelian varieties in these Ekedahl-Oort strata of $\mathcal{A}_5$, but these abelian varieties do not admit an action of $\mathcal{O}_K$ of signature $(3,2)$ with compatible polarization.
\end{remark}

\section{Explicit construction of a point} \label{s:construction}

Sections \ref{s:p-rank} and \ref{s:SMV} provide an answer to  Question~\ref{q} for all Ekedahl-Oort strata in $\mathcal{M}(3,2)$ except $\mathcal{M}(3,2)_{\gamma_{3,4}}$. It follows from Equation~\eqref{eq:p-rank 0 closure} that $\mathcal{M}(3,2)_{\gamma_{3,4}}$ intersects $\mathcal{M}(3,2)^{\beta_1}$ and is in fact dense in it. In this section, we show that $\mathcal{M}(3,2)_{\gamma_{3,4}}$ also intersects the supersingular locus, by explicitly constructing a point in the intersection $\mathcal{M}(3,2)_{\gamma_{3,4}} \cap \mathcal{M}(3,2)^{\ss}$. To do this, we first use $p$-adic  Dieudonn\'e theory to construct the $p$-divisible groups of these points; we then use the $p$-divisible groups and  Rapoport-Zink uniformization to construct points of $\mathcal{M}(3,2)$.

Let $\breve{\mathbx{Z}}_p = W(\kk)$ be the ring of Witt vectors of $\kk$, denote by $\mathrm{Frob} = W(\mathrm{Frob}_{\kk})$ the lift of $\mathrm{Frob}_{\kk}$ to $\breve{\mathbx{Z}}_p$, and let $\breve{\mathbx{Q}}_p = \breve{\mathbx{Z}}_p[\frac{1}{p}]$. Note that $\breve{\mathbx{Q}}_p$ is isomorphic to the completion of the maximal unramified extension of $\mathbx{Q}_p$, with ring of integers isomorphic to $\breve{\mathbx{Z}}_p$, and that the residue field of $\breve{\mathbx{Z}}_p$ is $\kk$. Let $\varphi_1$, $\varphi_2$ be the two embeddings of $\mathcal{O}_K$ into $\breve{\mathbx{Q}}_p$.

\begin{definition}\label{def:pdiv} For any scheme $S$ over $\kk$, a \emph{unitary $p$-divisible group of signature $(q-b,b)$} over $S$ is a triple $(X, \iota_X, \lambda_X)$, where

\begin{enumerate}
\item{$X$ is a $p$-divisible group over $S$ of height $2q$ and dimension $q$; }
\item{$\iota_X: \mathcal{O}_K \otimes_{\mathbx{Z}} \mathbx{Z}_p \rightarrow \mathrm{End}(X)$ is an action satisfying the signature $(q-b,b)$ condition
$$\mathrm{charpol}(\iota(a) \mid \mathrm{Lie}(X) ) = (T - \phi_1(a))^{q-b}(T-\phi_2(a))^b \in \breve{\mathbx{Z}}_p[T],$$
for all $a \in \mathcal{O}_K$;
}
\item{$\lambda_X: X \rightarrow X^\vee$ is a $p$-principal polarization, meeting the following $\mathcal{O}_K$-linearity condition
$$\lambda_X\circ \iota_X(a) = \iota_X(\overline{a})^\vee \circ \lambda_X,$$
for all $a \in \mathcal{O}_K$.
}
\end{enumerate}
\end{definition}

Over an algebraically closed field, we may study unitary $p$-divisible groups linear-algebraically.
 
\begin{definition}\label{def:pDD} A \emph{unitary $p$-adic Dieudonn\'e module of signature $(q-b,b)$} over $\kk$ 
is a tuple $(M, M = M_1 \oplus M_2, F, V, \langle \cdot , \cdot \rangle )$, where

\begin{enumerate}
\item{$M$ is a free $\breve{\mathbx{Z}}_p$-module of rank $2q$};\label{def:pDD1}
\item{$M = M_1 \oplus M_2$ is a decomposition into rank-$q$ summands};\label{def:pDD2}
\item{$F: M \rightarrow M$ is a $\mathrm{Frob}$-semilinear operator, $V: M \rightarrow M$ is a $\mathrm{Frob}^{-1}$-semilinear operator, with $F \circ V = V \circ F = p$};\label{def:pDD3}
\item{$\langle \cdot , \cdot \rangle$ is a perfect alternating $\breve{\mathbx{Z}}_p$-bilinear pairing on M such that
$\langle Fx, y \rangle = \langle x, Vy \rangle^\mathrm{Frob},$
for all $x,y \in M$};\label{def:pDD4}
\item{$\mathrm{dim}_{\kk}(M_1/F M_2) = q-b$ and $\mathrm{dim}_{\kk}(M_2/F M_1) = b$; }\label{def:pDD5}
\item {$F$ and $V$ are homogeneous of degree 1 with respect to the decomposition $M = M_1 \oplus M_2$;} \label{def:pDD6}
\item {$M_1$ and $M_2$ are each totally isotropic  with respect to $\langle \cdot , \cdot \rangle$. }\label{def:pDD7}
\end{enumerate}
\end{definition}

For the rest of this section, we consider signature $(3,2)$ and $q=5$. By contravariant Dieudonné theory, there is an anti-equivalence
of categories between unitary $p$-divisible groups and unitary $p$-adic Dieudonné modules, both of signature $(3,2)$ over $\kk$. When no confusion arises, we abbreviate unitary $p$-divisible groups as $X$ and unitary $p$-adic Dieudonné modules as $M$. 

\begin{lemma}\label{l:explicit} Let $G_{\gamma_{3,4}}$ be the $p$-torsion group scheme occurring as the $p$-torsion subgroup for points in the Ekedahl-Oort stratum $\mathcal{M}(3,2)_{\gamma_{3,4}}$. There exists a supersingular unitary $p$-divisible group $X_{3,4}$ of signature $(3,2)$ over $\kk$ such that $X_{3,4}[p] \cong G_{\gamma_{3,4}}$, respecting action and polarization. 
\end{lemma}

\begin{proof}
We aim to construct a supersingular unitary $p$-divisible group $X_{3,4}$ of signature $(3,2)$ such that $X_{3,4}[p] \cong G_{\gamma_{3,4}}$. Using contravariant  $p$-adic Dieudonné theory, it suffices to construct a unitary $p$-adic Dieudonné module $M$ of signature $(3,2)$ such that all of the slopes of the isocrystal $M[\frac{1}{p}]$ are equal to $\frac{1}{2}$ and such that $M/pM$ is isomorphic to the mod-$p$ Dieudonné module of $G_{\gamma_{3,4}}$.

Let $M$ be the free $\breve{\mathbx{Z}}_p$-module with basis $\{ e_i, f_i \}_{1\leq i \leq 5}$. Let $M_1$ be the submodule spanned by $\{e_i\}_{1\leq i \leq 5}$ and let $M_2$ be the submodule spanned by $\{f_i\}_{1\leq i \leq 5}$. Define $F$ (resp. $V$) as the $\mathrm{Frob}$-semilinear (resp. $\mathrm{Frob}^{-1}$-semilinear) operator defined on the basis of $M$ as in Table \ref{Table34}  below. 

\renewcommand{\arraystretch}{1.4}
\begin{table}[h]
    \begin{tabular}{|c | c | c | c|}
    \hline
    $F(e_1) = f_5$ & $V(e_1) = pf_2$ & $F(f_1) = -pe_5$ & $V(f_1) = e_2$     \\
    \hline
    $F(e_2) = pf_1$ & $V(e_2) = f_3$ &  $F(f_2) = e_1$ & $V(f_2) = pe_3$    \\
    \hline
    $F(e_3) = f_2$ & $V(e_3) = f_4$  & $F(f_3) = pe_2$ & $V(f_3) = pe_4$   \\
    \hline
    $F(e_4) = f_3$ & $V(e_4) = pf_5$ & $F(f_4) = pe_3$  & $V(f_4) = e_5$   \\
    \hline
    $F(e_5) = pf_4$ & $V(e_5) = -f_1$ &  $F(f_5) = e_4$ & $V(f_5) = pe_1$    \\
    \hline
    \end{tabular}
    \caption{$F$ and $V$ on $M$ for $\gamma_{3,4}$.}\label{Table34} 
\end{table}

We define an alternating pairing $\langle \cdot , \cdot \rangle$ on $M$ by the condition that $\langle e_i, f_j \rangle = (-1)^{i-1} \delta_{ij}$ 
and claim that $(M, M = M_1 \oplus M_2, F, V, \langle \cdot , \cdot \rangle )$ is a unitary $p$-adic Dieudonné module of signature $(3,2)$ since it satisfies each condition of Definition \ref{def:pDD}. Indeed:

\begin{itemize}[leftmargin=2em]

\item{$M_1$ and $M_2$ both have rank 5, so $M = M_1 \oplus M_2$ has rank $10=2q$.}\hfill {\footnotesize (Conditions \eqref{def:pDD1} \& \eqref{def:pDD2})}
\item{The operators $F$ and $V$ are homogeneous of degree 1 with respect to the decomposition $M = M_1 \oplus M_2$. Let $\mathbf{A}_F$ and $\mathbf{A}_V$ be the matrices given by the action of $F$ and $V$, respectively, on the chosen basis. Since $\mathbf{A}_F$ and $\mathbf{A}_V$ have integer entries, to check that $F \circ V = V \circ F = p$, it suffices to verify that $\mathbf{A}_F \mathbf{A}_V = \mathbf{A}_V \mathbf{A}_F = p \mathrm{Id}$, which is true by construction.}\hfill {\footnotesize(Conditions \eqref{def:pDD3} \& \eqref{def:pDD6})}
\item{The condition that  $\langle e_i, f_j \rangle = (-1)^{i-1}\delta_{ij}$  extends uniquely to a perfect alternating $\breve{\mathbx{Z}}_p$-bilinear pairing on $M$, and under this pairing $M_1$ and $M_2$ are each totally isotropic. 

Let $\mathbf{B}$ be the matrix of this alternating form, and note that $\mathbf{B}$ has integer entries. To check that
$\langle Fx, y \rangle = \langle x, Vy \rangle^\mathrm{Frob},$
 it suffices to verify that $\mathbf{A}_F^T \mathbf{B} = \mathbf{B} \mathbf{A}_V$, which can be verified using the description of $F$ and $V$ in Table~\ref{Table34} and the definition of $\langle \cdot , \cdot \rangle$.} \hfill {\footnotesize(Conditions \eqref{def:pDD4} \& \eqref{def:pDD7})}
 
\item{From the definition of $F$, we have $M_1/ F M_2 \cong \mathrm{Span}_{\kk} \{ e_2, e_3, e_5  \}$ and $M_2/ F M_1 \cong \mathrm{Span}_{\kk} \{ f_1, f_4  \}$. In particular, $\mathrm{dim}_{\kk}(M_1/F M_2) = 3$ and $\mathrm{dim}_{\kk}(M_2/F M_1) = 2$.}\hfill {\footnotesize(Condition \eqref{def:pDD5})} 
\end{itemize}

By contravariant Dieudonné theory, $M$ defines a unitary $p$-divisible group $X_{3,4}$ of signature $(3,2)$.

We use \cite[Lemma 6.12]{zink68cartier} to compute the slopes of the isocrystal $M[\frac{1}{p}]$. Note that for any positive integer $m$, we find that
\[F^{10m}(M) = p^{5m}M, \]
and so $\frac{1}{10m}\mathrm{max}\{k \in \mathbx{Z}: F^{10m}M\subset p^kM\} = \frac{1}{2}.$ Therefore, 
\[ \lim_{n \to \infty} \tfrac{1}{n}\max\{k \in \mathbx{Z}: F^nM \subset p^kM\} = \tfrac{1}{2}, \]
and by  \cite{zink68cartier} all slopes of $M[\frac{1}{p}]$ are equal to $\frac{1}{2}$. Accordingly, $X_{3,4}$ is supersingular.

Finally, we show $X_{3,4}[p]=G_{\gamma_{3,4}}$. Define $N$ as $M/pM\cong \mathrm{Span}_{\kk} \{ e_i, f_i \}_1^5$, with splitting $N = N_1 \oplus N_2$ and $F$ and $V$ operators induced from those on $M$. We will follow the procedure of \cite[3.5]{Moonen2001} to compute the permutation $\omega \in \mathbf{W}(3,2)$ associated to $N$. Following from the definition of $F$ and $V$, the Dieudonné module $N$ has final filtration
\begin{gather*}
    0 \subset \langle f_3 \rangle \subset \langle e_4, f_3 \rangle \subset \langle e_4, f_3, f_5 \rangle \subset \langle e_1, e_4, f_3,  f_5 \rangle \subset \langle e_1, e_4, f_2, f_3, f_5  \rangle \subset  \langle e_1, e_2, e_4, f_2, f_3, f_5 \rangle \\ 
    \subset \langle e_1, e_2, e_4, f_1, f_2, f_3, f_5  \rangle \subset \langle e_1, e_2, e_4, e_5, f_1, f_2, f_3, f_5 \rangle \subset \langle e_1, e_2, e_4,  e_5, f_1, f_2, f_3, f_4, f_5 \rangle \subset N.
\end{gather*}

Intersecting with $N_1$ gives the filtration $C_{1,\bullet}$
$$0 \subset \langle e_4 \rangle \subset \langle e_1, e_4 \rangle \subset \langle e_1, e_2, e_4 \rangle \subset \langle e_1, e_2,  e_4, e_5 \rangle \subset N_1.$$

The function $\eta_1(j)= \dim(C_{1,j} \cap N[F])$ is then given by
\[\eta_1(1) = \eta_1(2) = 0, \hspace{1cm}
\eta_1(3) = 1,\hspace{1cm} \eta_1(4) = \eta_1(5) = 2.\]

The permutation $\omega$ corresponding to $\eta$ is $(1,3)(2,4)$. As $\gamma_{3,4}$ is also equal to $(1,3)(2,4)$, it follows that $X_{3,4}[p] \cong G_{\gamma_{3,4}}$, which finishes the proof.
\end{proof}

Recall that the supersingular locus $\mathcal{M}(3,2)^{\ss}$ is uniformized by a formal scheme called a Rapoport-Zink space. As a framing object, let $(\mathbx{X}, \iota_{\mathbx{X}}, \lambda_{\mathbx{X}} )$ 
 be a fixed supersingular unitary $p$-divisible group of signature $(3,2)$  over $\kk$.

\begin{definition}\label{def:RZ}
For any scheme $S$ over $\kk$, denote by $\mathcal{N}(3,2)(S)$ the set of isomorphism classes of tuples $(X, \iota_X, \lambda_X, \rho_X)$, where:
\begin{itemize}
\item{$(X, \iota_X, \lambda_X)$ is a unitary $p$-divisible group of signature $(3,2)$ over $S$;}
\item{$\rho_X: X \rightarrow \mathbx{X}$ is an $\mathcal{O}_K$-linear quasi-isogeny identifying  $\lambda_X$ and $\lambda_{\mathbx{X}}$ up to scaling in $\mathbx{Q}_p^\times$.}
    
\end{itemize}

By \cite{RZ}, the functor defined above is represented by a formal scheme over $\kk$ which is locally formally of finite type; we will also denote the underlying reduced scheme of this representing object (a ``signature $(3,2)$ unitary Rapoport-Zink space") as $\mathcal{N}(3,2)$.

\end{definition}

\begin{proposition}\label{p:gamma34}
  The Ekedahl-Oort stratum $\mathcal{M}(3,2)_{\gamma_{3,4}}$ intersects the supersingular locus.
\end{proposition}

\begin{proof}

By the uniformizaton theorem of Rapoport and Zink \cite{RZ}, there exist groups $\{ \Gamma_j \}_{j=1}^n$ (arising as subgroups of the $\mathbx{Q}_p$-points of the algebraic group defining the automorphisms of $\mathbx{X}$, and depending on the level structure implicit in the definition of $\mathcal{M}(3,2)$) such that there is an isomorphism of schemes over $\kk$
\begin{align*}
\bigsqcup_{j=1}^n \mathcal{N}(3,2)/ \Gamma_j &\cong \mathcal{M}(3,2)^{\ss}.
\intertext{In particular, there is a surjection of $\kk$-points}
\bigsqcup_{j=1}^n \mathcal{N}(3,2)(\kk)  &\twoheadrightarrow \mathcal{M}(3,2)^{\ss}(\kk).
\end{align*}
Let $G_{\gamma_{3,4}}  $ be the $p$-torsion group scheme (with extra structure) occurring as $p$-torsion subgroup for points in the Ekedahl-Oort stratum $\mathcal{M}(3,2)_{\gamma_{3,4}}$. By Lemma \ref{l:explicit}, there exists a supersingular unitary $p$-divisible group $X_{3,4}$  of signature $(3,2)$ over $\kk$ such that $X_{3,4}[p] \cong G_{\gamma_{3,4}} $ . 

By \cite[Lemma 6.1]{VollaardWedhorn}, there exists a quasi-isogeny $\rho_{X_{3,4}}: X_{3,4} \rightarrow \mathbx{X}$ that is $\mathcal{O}_K$-linear and identifies the polarizations, up to $\mathbx{Q}_p^\times$-scaling. That is, $(X_{3,4}, \iota_{X_{3,4}}, \lambda_{X_{3,4}}, \rho_{X_{3,4}})$ defines a $\kk$-point of $\mathcal{N}(3,2)$. Let $(A_{3,4}, \iota_{A_{3,4}}, \lambda_{A_{3,4}}, \xi_{A_{3,4}} )$ be the image of $(X_{3,4}, \iota_{X_{3,4}}, \lambda_{X_{3,4}}, \rho_{X_{3,4}})$ in $\mathcal{M}(3,2)^{\ss}(\kk)$. 

Since $A_{3,4}[p] \cong X_{3,4}[p] \cong G_{\gamma_{3,4}}$, as the $p$-torsion group schemes equipped with extra structure, the $\kk$-point $(A_{3,4}, \iota_{A_{3,4}}, \lambda_{A_{3,4}}, \xi_{A_{3,4}})$ of $\mathcal{M}(3,2)$ lies in the Ekedahl-Oort stratum indexed by $\gamma_{3,4}$. That is, $(A_{3,4}, \iota_{A_{3,4}}, \lambda_{A_{3,4}}, \xi_{A_{3,4}} )$ is an explicit point in the intersection $\mathcal{M}(3,2)_{\gamma_{3,4}}\cap \mathcal{M}(3,2)^{\ss}$. 
\end{proof}

\begin{remark} \label{r:Coxeter}
We remark that there is a shorter but less elegant proof of Proposition~\ref{p:gamma34}. By \cite[Theorem A]{goertzhe}, $\mathcal{M}(3,2)$ is not of Coxeter type, so Ekedahl-Oort stratification does not refine the Newton stratification. There exists at least one Ekedahl-Oort stratum that intersects at least two Newton strata. It was already shown that every Ekedahl-Oort stratum apart from $\mathcal{M}(3,2)_{\gamma_{3,4}}$ is completely contained in one Newton stratum. Hence, it must be that $\mathcal{M}(3,2)_{\gamma_{3,4}}$ intersects $\mathcal{M}(3,2)^{\ss}$ as well as $\mathcal{M}(3,2)^{\beta_1}$. However, we believe that the explicit construction is more insightful.
\end{remark}

We conclude this section with a final result regarding the interaction between the Ekedahl-Oort stratum $\mathcal{M}(3,2)_{\gamma_{3,4}}$ and the supersingular locus of the Newton stratification. We saw in Proposition~\ref{p:gamma34} that these strata intersect, and a natural next question to ask is: what are the dimensions of the irreducible components of this intersection? The following corollary answers this.

\begin{corollary}\label{c:dimensionresult}
Every irreducible component of $\mathcal{M}(3,2)_{\gamma_{3,4}} \cap \mathcal{M}(3,2)^{\ss}$ has dimension $3$.

\end{corollary} 
\begin{proof}
First note that $\mathcal{M}(3,2)_{\gamma_{1,5}}$ is contained in the supersingular locus by Lemma~\ref{c:generic}. It follows that the closure $\overline{\mathcal{M}(3,2)_{\gamma_{1,5}}}$ is too, since the supersingular locus is closed. The strata $\mathcal{M}(3,2)_{\gamma_{2,3}}$ and $\mathcal{M}(3,2)_{\gamma_{3,4}}$ also intersect the supersingular locus by Lemma~\ref{c:generic} and Proposition~\ref{p:gamma34}. However, it is shown in \cite[Figure 3-3]{wooding_2016} that these Ekedahl-Oort strata are disjoint from $\overline{\mathcal{M}(3,2)_{\gamma_{1,5}}}$. Thus,
$$\overline{\mathcal{M}(3,2)_{\gamma_{1,5}}} \subsetneq \mathcal{M}(3,2)^{\ss}.$$ 

By the purity theorem \cite[Theorem 4.1]{DeJongOort}, every irreducible component of $\mathcal{M}(3,2)^{\ss} \setminus \overline{\mathcal{M}(3,2)_{\gamma_{1,5}}}$ has dimension $3$. Since $\mathcal{M}(3,2)_{\gamma_{3,4}}$ is the only Ekedahl-Oort stratum that intersects $\mathcal{M}(3,2)^{\ss} \setminus \overline{\mathcal{M}(3,2)_{\gamma_{1,5}}}$ and has dimension at least $3$, it follows that $\mathcal{M}(3,2)_{\gamma_{3,4}}$ must be dense in $\mathcal{M}(3,2)^{\ss} \setminus \overline{\mathcal{M}(3,2)_{\gamma_{1,5}}}$. Hence, every irreducible component of $\mathcal{M}(3,2)_{\gamma_{3,4}} \cap \mathcal{M}(3,2)^{\ss}$ has dimension 3.
\end{proof}

Note that Corollary \ref{c:dimensionresult} also follows from the work of Shimada \cite{shimada2024supersingular}, by very different methods. We include the self-contained argument above for the sake of completion.

\section{Classification of the interaction between the stratifications}

We are now in a position to state and prove the main result.
\begin{theorem} \label{thm:main}
The interaction between the Ekedahl-Oort stratification and the Newton stratification of $\mathcal{M}(3,2)$ is as follows:
\begin{enumerate}[label = (\roman*)]
    \item The Ekedahl-Oort strata $\mathcal{M}(3,2)_{\gamma_{1,2}}$, $\mathcal{M}(3,2)_{\gamma_{1,3}}$, $\mathcal{M}(3,2)_{\gamma_{1,4}}$, $\mathcal{M}(3,2)_{\gamma_{1,5}}$ and $\mathcal{M}(3,2)_{\gamma_{2,3}}$ are contained in the Newton stratum $\mathcal{M}(3,2)^{\ss}$; \label{item:ss}
    \item The Ekedahl-Oort stratum $\mathcal{M}(3,2)_{\gamma_{2,4}}$ is contained in the Newton stratum $\mathcal{M}(3,2)^{\beta_1}$; \label{item:24}
    \item The Ekedahl-Oort stratum $\mathcal{M}(3,2)_{\gamma_{3,4}}$ intersects both $\mathcal{M}(3,2)^{\ss}$ and $\mathcal{M}(3,2)^{\beta_1}$; \label{item:34}
    \item The Ekedahl-Oort strata $\mathcal{M}(3,2)_{\gamma_{2,5}}$ and $\mathcal{M}(3,2)_{\gamma_{3,5}}$ are contained in the Newton stratum $\mathcal{M}(3,2)^{\beta_2}$; \label{item:p-rank2}\vspace*{-1em}
    \item The Ekedahl-Oort stratum $\mathcal{M}(3,2)_{\gamma_{4,5}}$ equals the Newton stratum $\mathcal{M}(3,2)^{\ord}$. \label{item:ord}
\end{enumerate}
\end{theorem}
\begin{proof}
Item~\ref{item:ss} is obtained by combining Lemma~\ref{l:Hoeve} and Corollary~\ref{c:generic}. Item~\ref{item:24} is the content of Lemma~\ref{l:gamma24}. Next, Item~\ref{item:34} follows from combining Proposition~\ref{p:gamma34} and Equation~\eqref{eq:p-rank 0 closure}. Finally, Items~\ref{item:p-rank2} and \ref{item:ord} follow from Lemma~\ref{l:p-rank}.
\end{proof}


\renewcommand\bibpreamble{\vspace*{-0.1\baselineskip}}

\bibliography{citations.bib}{}

\newcommand{\etalchar}[1]{$^{#1}$}
\providecommand{\bysame}{\leavevmode\hbox to3em{\hrulefill}\thinspace}
\providecommand{\MR}{\relax\ifhmode\unskip\space\fi MR }
\providecommand{\MRhref}[2]{%
  \href{http://www.ams.org/mathscinet-getitem?mr=#1}{#2}
}
\providecommand{\href}[2]{#2}
\begin{thebibliography}{ABF{\etalchar{+}}24}

\bibitem[ABF{\etalchar{+}}24]{RNT}
Emerald Andrews, Deewang Bhamidipati, Maria Fox, Heidi Goodson, Steven Groen, and Sandra Nair, \emph{Ekedahl-{O}ort strata in the {$\mathsf{GU}(q-2,2)$} {S}himura variety}, arXiv (2024), 2405.04464.

\bibitem[BW06]{bultel2006congruence}
Oliver B{\"u}ltel and Torsten Wedhorn, \emph{Congruence relations for shimura varieties associated to some unitary groups}, Journal of the Institute of Mathematics of Jussieu \textbf{5} (2006), no.~2, 229--261.

\bibitem[dJO00]{DeJongOort}
Aise~Johan de~Jong and Frans Oort, \emph{Purity of the stratification by newton polygons}, Journal of the American Mathematical Society \textbf{13} (2000), no.~1, 209–241.

\bibitem[GH15]{goertzhe}
Ulrich G\"{o}rtz and Xuhua He, \emph{Basic loci of {C}oxeter type in {S}himura varieties}, Camb. J. Math. \textbf{3} (2015), no.~3, 323--353. \MR{3393024}

\bibitem[Har07]{Harashita_slope}
Shushi Harashita, \emph{Ekedahl-oort strata and the first newton slope strata}, Journal of Algebraic Geometry \textbf{16} (2007), no.~1, 171--199.

\bibitem[Har10]{Harashita_polygon}
Shushi Harashita, \emph{Generic {Newton} polygons of {Ekedahl}-{Oort} strata: {Oort}'s conjecture}, Ann. Inst. Fourier \textbf{60} (2010), no.~5, 1787--1830 (English).

\bibitem[Hoe09]{Hoeve2009}
Maarten Hoeve, \emph{Ekedahl-oort strata in the supersingular locus}, Journal of the London Mathematical Society \textbf{81} (2009), no.~1, 129--141.

\bibitem[HP14]{howard2014supersingular}
Benjamin Howard and Georgios Pappas, \emph{On the supersingular locus of the {$GU(2,2)$} shimura variety}, Algebra Number Theory \textbf{8} (2014), no.~7, 1659--1699.

\bibitem[Kot92]{Ko}
Robert~E. Kottwitz, \emph{Points on some {S}himura varieties over finite fields}, J. Amer. Math. Soc. \textbf{5} (1992), no.~2, 373--444. \MR{1124982}

\bibitem[Moo01]{Moonen2001}
Ben Moonen, \emph{{Group Schemes with Additional Structures and Weyl Group Cosets}}, pp.~255--298, Birkh{\"a}user Basel, Basel, 2001.

\bibitem[Moo04]{MoonenST}
\bysame, \emph{Serre-{T}ate theory for moduli spaces of {PEL} type}, Ann. Sci. \'Ecole Norm. Sup. (4) \textbf{37} (2004), no.~2, 223--269. \MR{2061781}

\bibitem[Oor05a]{Oort05minimal}
Frans Oort, \emph{Minimal p-divisible groups}, Annals of {M}athematics \textbf{161} (2005), 1021--1036.

\bibitem[Oor05b]{Oort05simple}
\bysame, \emph{Simple p-kernels of p-divisible groups}, Advances in {M}athematics \textbf{198} (2005), 275--310.

\bibitem[RR96]{rapoport1996classification}
Michael Rapoport and Melanie Richartz, \emph{On the classification and specialization of $ f $-isocrystals with additional structure}, Compositio Mathematica \textbf{103} (1996), no.~2, 153--181.

\bibitem[RZ96]{RZ}
M.~Rapoport and Th. Zink, \emph{Period spaces for {$p$}-divisible groups}, Annals of Mathematics Studies, vol. 141, Princeton University Press, Princeton, NJ, 1996. \MR{1393439}

\bibitem[Shi24]{shimada2024supersingular}
Ryosuke Shimada, \emph{On the supersingular locus of the {$\mathrm{GU}(2,n-2)$} {S}himura variety}, arXiv (2024), 2410.05110.

\bibitem[Vol10]{vollaard}
Inken Vollaard, \emph{The supersingular locus of the {S}himura variety for {${\rm GU}(1,s)$}}, Canad. J. Math. \textbf{62} (2010), no.~3, 668--720. \MR{2666394}

\bibitem[VW11]{VollaardWedhorn}
Inken Vollaard and Torsten Wedhorn, \emph{The supersingular locus of the {S}himura variety of {$GU(1,n-1)$} {II}}, Inventiones {M}athematicae \textbf{184} (2011), 591--627.

\bibitem[VW13]{ViehmannWedhorn2013}
Eva Viehmann and Torsten Wedhorn, \emph{Ekedahl-{O}ort and {N}ewton strata for {S}himura varieties of {PEL} type}, Math. Ann. \textbf{356} (2013), no.~4, 1493--1550. \MR{3072810}

\bibitem[Woo16]{wooding_2016}
Amy Wai Ling~Jane Wooding, \emph{{The Ekedahl-Oort stratification of unitary Shimura varieties}}, Ph.D. thesis, McGill University, 2016.

\bibitem[Zin84]{zink68cartier}
Thomas Zink, \emph{Cartiertheorie kommutativer former gruppen}, Teubner-Texte zur Mathematik, 1984.

\end{thebibliography}
\bibliographystyle{amsalpha}

\end{document}